\documentclass[10pt,reqno]{amsart}
\usepackage{amsmath}
\usepackage{amssymb}
\usepackage{amsthm}
\usepackage{eepic,epic}
\usepackage{epsfig}
\usepackage{graphicx}
\newlength{\defbaselineskip}
\newcommand{\setlinespacing}[1]%
           {\setlength{\baselineskip}{#1 \defbaselineskip}}

\numberwithin{equation}{section}

\newtheorem{thm}{Theorem}[section]

\newtheorem{lem}[thm]{Lemma}
\newtheorem{prop}[thm]{Proposition}

\theoremstyle{definition}
\newtheorem{defn}[thm]{Definition}

\theoremstyle{remark}
\newtheorem{rem}[thm]{Remark}
\numberwithin{equation}{section}

\begin{document}

\title[Critical inhomogeneous Hartree equations]
{On well-posedness for inhomogeneous Hartree equations in the critical case}

\author{Seongyeon Kim}

\thanks{This work was supported by a KIAS Individual Grant (MG082901) at Korea Institute for Advanced Study and the POSCO Science Fellowship of POSCO TJ Park Foundation.}


\address{School of Mathematics, Korea Institute for Advanced Study, Seoul 02455, Republic of Korea}
\email{synkim@kias.re.kr}

\begin{abstract}
We study the well-posedness for the inhomogeneous Hartree equation $i\partial_t u + \Delta u = \lambda(I_\alpha \ast |\cdot|^{-b}|u|^p)|x|^{-b}|u|^{p-2}u$ in $H^s$, $s\ge0$.
Until recently, its well-posedness theory has been intensively studied, focusing on solving the problem for the critical index $p=1+\frac{2-2b+\alpha}{n-2s}$ with $0\le s \le 1$, but the case $1/2\leq s \leq 1$ is still an open problem.
In this paper, we develop the well-posedness theory in this case, especially including the energy-critical case.
To this end, we approach to the matter based on the Sobolev-Lorentz space which can lead us to perform a finer analysis for this equation. This is because it makes it possible to control the nonlinearity involving the singularity $|x|^{-b}$ as well as the Riesz potential $I_\alpha$ more effectively.

\end{abstract}

\maketitle

\section{Introduction}
In this paper, we are concerned with the Cauchy problem for the inhomogeneous Hartree equation
\begin{equation}\label{HE}
	\begin{cases}
		i \partial_t u + \Delta u = \lambda(I_\alpha \ast |\cdot|^{-b} |u|^p)|x|^{-b} |u|^{p-2}u, \quad (x,t) \in \mathbb{R}^n \times \mathbb{R}, \\
		u(x,0)= u_0(x),
	\end{cases}
\end{equation}
where $p\ge 2$, $b>0$ and $\lambda = \pm1$.
Here, the case $\lambda=1$ is \textit{defocusing}, while the case $\lambda=-1$ is \textit{focusing}.
The Riesz potential $I_\alpha$ is defined on $\mathbb{R}^n$ by 
$$I_{\alpha}:=\frac{\Gamma(\frac{n-\alpha}{2})}{\Gamma({\frac \alpha2})\pi^{\frac n2} 2^{\alpha} |\cdot|^{n-\alpha}}, \quad 0<\alpha<n.$$
The problem \eqref{HE} arises in the physics of laser beams and of multiple-particle systems \cite{GM, MPT}.
The homogeneous case $b=0$ in \eqref{HE} is called the Hartree equation (or Choquard equation) and has several physical origins such as quantum mechanics and Hartree-Fock theory. In particular, if $b=0$ and $p=2$, it models the dynamics of boson stars, where the potential is the Newtonian gravitational potential in the appropriate physical units.

Before looking at known results for the Cauchy problem \eqref{HE}, we first recall the critical Sobolev index.
Note that if $u(x,t)$ is a solution of \eqref{HE} so is $$u_\delta(x,t)=\delta^{\frac{2-2b+\alpha}{2(p-1)}} u(\delta x, \delta^2 t),$$ with the rescaled initial data $u_{\delta,0}=u_{\delta}(x,0)$ for all $\delta>0$.
Then we see that
\begin{equation*}
	\|u_{\delta,0}\|_{\dot H^s}=\delta^{s-\frac n2 +\frac{2-2b+\alpha}{2(p-1)}}\|u_0\|_{\dot H^s},
\end{equation*}
from which the critical Sobolev index is given by $s_c=\frac{n}{2}-\frac{2-2b+\alpha}{2(p-1)}$ (alternatively $p=1+\frac{2-2b+\alpha}{n-2s_c}$) which determines the scale-invariant Sobolev space $\dot H^{s_c}$.
In this regard, one can divide the matter into three cases, $s_c=0$, $s_c=1$ and $0<s_c<1$.
The case $s_c=0$ (alternatively $p=p_\ast:=1+\frac{\alpha+2-2b}{n}$) is referred to as the mass-critical (or $L^2$-critical). If $s_c=1$ (alternatively $p=p^\ast:=1+\frac{2-2b+\alpha}{n-2}$) the problem is called the energy-critical (or $H^1$-critical), and it is known as the mass-supercritical and energy-subcritical if $0<s_c<1$. 
Therefore, we are only interested in the cases $0\leq s_c \leq 1$ in this paper.

The well-posedness theory of the Hartree equation ($b=0$ in \eqref{HE}) has been extensively studied over the past few decades and is well understood. (See, for example, \cite{MXZ, CHO, GW, Saa} and references therein.)
For the inhomogeneous model \eqref{HE} drawn attention in recent several years, its well-posedness theory has been developed in various attempts to treat the singularity $|x|^{-b}$ in the nonlinearity which further complicates the problem.
The equation \eqref{HE} was first studied by Alharbi and Saanouni \cite{AS} using an Gargliardo-Nirenberg type inequality adapted for the nonlinearity.
They proved that \eqref{HE} is locally well-posed in $L^2$ if $2\leq p < p_\ast$ and in $H^1$ if $2\leq p<p^\ast$.
After that, Saanouni and Alharbi \cite{SA} treated the intermediate case in the sense that \eqref{HE} is locally well-posed in $\dot H^1 \cap \dot H^{s_c}$, $0<s_c<1$, if $2\leq p < p^\ast$, but this does not imply the inter-critical case $H^{s_c}$.
For related results on the scattering theory, see also \cite{SX,SP,X}.
Recently, in \cite{KLS1}, the well-posedness in the critical case was partially solved by introducing a weighted space which makes it possible to deal with the nonlinearity efficiently.
Indeed, they proved that \eqref{HE} is locally well-posed in $H^{s_c}$ for $0\leq s_c <1/2$, and in $\dot H^{s_c}$ for $-1/2<s_c<0$ which is the critical case below $L^2$.
However, the remaining case $1/2 \leq s_c \leq 1$ is still an open problem.

The main contribution of this paper is to completely settle the critical case of $0\leq s_c\leq 1$, especially including the energy-critical case $H^1$. 
To this end, we approach to the matter based on the following Sobolev-Lorentz spaces defined as 
$$W_{p,q}^s(\mathbb{R}^n):= (I-\Delta)^{s/2}L^{p,q}(\mathbb{R}^n),$$
where $L^{p,q}$ denote the Lorenz spaces. We will give the details in Section \ref{se2}.
These spaces can lead us to perform a finer analysis for the inhomogeneous Hartree equation \eqref{HE}.
This is because it makes it possible to control the nonlinearity involving the singularity $|x|^{-b}$ as well as the Riesz potential $I_\alpha$ more effectively. Particularly, the power function $|x|^{-\gamma}$, $\gamma>0$ belongs to $L^{p,\infty}$ if $p=n/\gamma$, but is not in $L^p$.

Before stating our results, we introduce the following definition. 
\begin{defn}
Let $n\ge3$. We say that $(q,r)$ is an admissible pair if it satisfies 
\begin{equation}
2\leq q \leq \infty,\quad 2\leq r \leq \frac{2n}{n-2} \quad \text{and} \quad \frac{2}{q}+\frac{n}{r}=\frac{n}{2}.
\end{equation}
\end{defn} 

Our main result is the following well-posedness in the critical case $p=1+\frac{2-2b+\alpha}{n-2s}$ when $0\leq s \leq 1$. 

\begin{thm}\label{Hloc}
Let $n\ge3$ and $0\leq s \leq 1$.
Assume that
\begin{equation}\label{as-1}
\max\Big\{\frac{n-2}{3}, n-4\Big\}<\alpha<n
\end{equation}
and
\begin{equation}\label{as-2}
\max\Big\{0, \frac{\alpha}{2}+1-\frac{(n-2s)(n+4+\sqrt{9n^2-8n+16})}{8(n-2)}\Big\}<b\leq \frac{\alpha}{2}+1-\frac{n-2s}{2}.
\end{equation}
Then for $u_0 \in H^s(\mathbb{R}^n)$ there exist $T>0$ and a unique solution $u \in C([0,T); H^s) \cap L^q ([0,T);W_{r,2}^s)$ to the problem \eqref{HE} with $p=1+\frac{2-2b+\alpha}{n-2s}$ if $(q,r)$ is any admissible pair satisfying 
\begin{equation}\label{as-r}
\max\Big\{s, s+\frac{\alpha-b}{p},\frac{n}{2}-\frac{2}{2p-1}\Big\}<\frac{n}{r}<\min\Big\{\frac{s(p-1)-b+n}{p} ,\frac{n}{2}-\frac{1}{2p-1}\Big\}.
\end{equation}
Furthermore, the continuous dependence on the initial data holds.
\end{thm}

\begin{rem}
In \cite{KLS1}, the authors showed that \eqref{HE} is locally well-posed in $H^s$ for $0\leq s<1/2$, if $p=1+\frac{2-2b+\alpha}{n-2s}$, 
\begin{equation}\label{pre}
n-2<\alpha<n \quad \text{and} \quad \max\Big\{0,\frac{\alpha-n}{2}+\frac{(n+2)s}{n}\Big\} <b \leq \frac{\alpha}2+1-\frac{n-2s}{2}.
\end{equation}
Since $\max\{\frac{n-2}3,n-4\}<n-2$, the range of $\alpha$ in the theorem contains the one of $\alpha$ in \eqref{pre}.
For fixed $\alpha$ and $s$ with $0\leq s<1/2$, since the lower bound of $b$ in \eqref{as-2} is smaller than the lower one of $b$ in \eqref{pre}, the range of $b$ in the theorem also contains the one of $b$ in \eqref{pre}.
Therefore, the theorem not only fills completely the gap $1/2\leq s\leq 1$ but also extends the existing ranges of $\alpha$ and $b$.
\end{rem}

We also provide the small data global well-posedness and the scattering results as follows:
\begin{thm}\label{Hglo}
Under the same conditions as in Theorem \ref{Hloc} and the smallness assumption on $\|u_0\|_{H^s}$, the local solution extends globally in time with 
\begin{equation}
    u \in C([0,\infty); H^s) \cap L^q ([0,\infty);W_{r,2}^s).
\end{equation}
Furthermore, the solution scatters in $ H^s$, i.e., there exists $\phi \in H^s$ such that
\begin{equation}
    \lim_{t\rightarrow \infty} \|u(t) - e^{it\Delta} \phi\|_{H^s}=0.
\end{equation}
\end{thm}

\

\noindent\textit{Outline of the paper.}
In Section \ref{se2} we introduce some function spaces and relevant properties that we need.
To prove the theorems, in Section \ref{se3}, we obtain nonlinear estimates (Proposition \ref{prop}) relying on Sobolev type inequalities in Lorentz spaces.
These estimates will play a key role in Section \ref{se4}, when proving the well-posedness results by applying the contraction mapping principle.

\

Throughout this paper, the letter $C$ stands for a positive constant which may be different at each occurrence.
We also denote $A\lesssim B$ to mean $A\leq CB$ with unspecified constants $C>0$.

\section{Preliminaries}\label{se2}
In this section we introduce some function spaces and relevant properties which will be used in this paper. 
We also recall the Strichartz estimates in the Lorentz spaces.
\subsection{Lorentz spaces}
Let $f$ be a measurable function on $\mathbb{R}^n$. 
The decreasing rearrangement of $f$, denoted $f^\ast$, is defined by 
$$f^*(s)= \inf \{\lambda>0: d_f(\lambda)\leq s\}$$
where $d_{f}(\lambda)$ is the distribution function of $f$ defined as 
$$d_f(\lambda)=\big|\{x\in\mathbb{R}^n: |f(x)|>\lambda\}\big|.$$
Here, $|A|$ denotes the Lebesgue measure of a set $A\subset \mathbb{R}^n$.
Then, for $1\leq p,q\leq\infty$, the Lorentz space $L^{p,q}(\mathbb{R}^n)$ is defined by
$$L^{p,q}(\mathbb{R}^n) = \{f: f \,\, \text{is measurable on $\mathbb{R}^n$}, \|f\|_{L^{p,q}(\mathbb{R}^n)}<\infty \}$$
where 
$$\|f\|_{L^{p,q}}=
\begin{cases}
&\Big(\int_0^{\infty} (s^{\frac1p} f^*(s))^{q}\frac{ds}{s}\Big)^{\frac1q} \quad \text{if} \quad q<\infty, \\
&\sup_{s>0} t^{\frac1p} f^*(s)\quad \text{if} \quad q=\infty.
\end{cases}$$
Note that $L^{p,q}$ is a quasi-Banach space since the quantity $\|\cdot\|_{L^{p,q}}$ does not satisfy the triangle inequality. 
When $1<p<\infty$ and $1\leq q \leq \infty$, the space $L^{p,q}$ can be normed to become a Banach space via an equivalent norm defined by
$$|||f|||_{L^{p,q}}=\begin{cases} \Big(\int_0^{\infty} \big(s^{\frac1p}f^{**}(s)\big)^q \frac{ds}{s}\Big)^{\frac1q} \quad \text{if}\quad q<\infty,\\
\sup_{s>0}s^{1/p}f^{**}(s)\quad \text{if} \quad q=\infty,
\end{cases}$$
where $f^{**}(t)=\frac1t\int_0^t f^*(s)ds,$ $t>0$.

We now recall some simple properties for the Lorentz spaces.
For any fixed $p$, the Lorentz spaces $L^{p,q}$ increase as the exponent $q$ increases. 
In other words, for $0<p\leq \infty$ and $0<q<r\leq\infty$, there exists $C>0$ such that 
\begin{equation}\label{emb}
\|f\|_{L^{p,r}}\leq C\|f\|_{L^{p,q}},
\end{equation}
which implies $L^{p,q}\subset L^{p,r}$.
Also, for $0<p,r<\infty$ and $0<q\leq\infty$ we have 
\begin{equation*}
\||f|^r\|_{L^{p,q}}=\|f\|^r_{L^{pr,qr}}.
\end{equation*}
In particular, the space $L^{p,\infty}$ includes the singularity $|x|^{-\gamma}$ if $p=\frac{n}{\gamma}$ (i.e., $|x|^{-\gamma}\in L^{\frac{n}{\gamma},\infty}$), but which does not belong to any Lebesgue space.

Finally, we introduce some well-known lemmas in Lorentz space which will be used in the next section.
\begin{lem}(H\"older's inequality)\label{H}
Let $1<p,p_1 , p_2 <\infty$ and $1\leq q, q_1, q_2 \leq \infty$. Then we have
\begin{equation}\label{Ho}
\|fg\|_{L^{p,q}} \leq C \|f\|_{L^{p_1,q_1}} \|g\|_{L^{p_2,q_2}},
\end{equation}
where 
$$\frac1p=\frac1{p_1}+\frac1{p_2} \quad\text{and}\quad \frac1q=\frac1{q_1}+\frac1{q_2}.$$
\end{lem}
\begin{lem}(Hardy-Littlewood-Sobolev inequality)\label{HLS}
Let $0<\alpha<n$ and $1<p<q<\infty$ with $1/q=1/p-\alpha/n$. If $1\leq r\leq \infty$, then
\begin{equation*}
\|I_\alpha \ast f\|_{L^{q,r}} \leq C \|f\|_{L^{p,r}}.
\end{equation*}
\end{lem}

\subsection{Sobolev-Lorentz spaces}
Now we give the definitions of the Sobolev-Lorentz spaces and present some useful properties which will be used (see also \cite{AT}).

Let $s\ge0$, $1<p<\infty$ and $1\leq q\leq\infty$.
The homogeneous Sobolev-Lorentz space ${\dot W}_{p,q}^s(\mathbb{R}^n)$ is defined as the space of functions $(-\Delta)^{s/2}f\in L^{p,q}(\mathbb{R}^n)$ equipped with the norm 
$$\|f\|_{{\dot W}_{p,q}^s(\mathbb{R}^n)}:=\big\|(-\Delta)^{s/2}f\big\|_{L^{p,q}}.$$
Also, the inhomogeneous Sobolev-Lorentz space $W_{p,q}^{s}(\mathbb{R}^n)$ is defined as the space of functions 
$(1-\Delta)^{s/2}f\in L^{p,q}(\mathbb{R}^n)$ equipped with the norm
$$\|f\|_{W^s_{p,q}}:=\|(1-\Delta)^{s/2}f\|_{L^{p,q}}\sim\|f\|_{L^{p,q}}+\|(-\Delta)^{s/2}f\|_{L^{p,q}}.$$
Here, $(-\Delta)^{s/2}$ and $(1-\Delta)^{s/2}$ are the Fourier multiplier operators defined by the multipliers $|\xi|^s$ and $(1+|\xi|^2)^{s/2}$, respectively.

The followings are useful inequalities for the Sobolev-Lorentz space that we need to prove our theorems.
\begin{lem}(Sobolev inequality, \cite[Theorem 2.4]{LR})\label{S}
Let $1<p<\infty$, $1\leq q \leq \infty$ and $0<s<n/p$ with $1/{p_1}=1/p-s/n.$
Then we have
\begin{equation}\label{s}
\|f\|_{L^{{p_1},q}} \leq C\|(-\Delta)^{s/2}f\|_{L^{p,q}}.
\end{equation}
\end{lem}

\begin{lem}(Fractional Leibniz rule, \cite[Theorem 6.1]{CUN})\label{FLR}
Let $s\ge0$ and $1<p,p_i<\infty$, $1\leq q,q_i\leq\infty$, $i=1,2,3,4$.
Then we have
\begin{equation*}
\|(-\Delta)^{s/2}(fg)\|_{L^{p,q}} \lesssim \|(-\Delta)^{s/2}f\|_{L^{p_1,q_1}}\|g\|_{L^{p_2,q_2}}+\|f\|_{L^{p_3,q_3}}\|(-\Delta)^{s/2}g\|_{L^{p_4,q_4}}
\end{equation*}
if 
$$\frac1p=\frac1{p_1}+\frac1{p_2}=\frac{1}{p_3}+\frac{1}{p_4}\quad \text{and} \quad \frac{1}{q}=\frac1{q_1}+\frac{1}{q_2}=\frac{1}{q_3}+\frac{1}{q_4}.$$
\end{lem}

\begin{lem}(Fractional chain rule, \cite[Lemma 2.4]{AT})\label{FCR}
Let $0\leq s\leq 1$ and $F\in C^1(\mathbb{C},\mathbb{C})$.
Then, for $1<p,p_1,p_2<\infty$, $1\leq q,q_1,q_2<\infty$,
we have 
\begin{equation*}
\|(-\Delta)^{s/2}F(f)\|_{L^{p,q}} \leq C \|F'(f)\|_{L^{p_1 ,q_1}} \|(-\Delta)^{s/2}f\|_{L^{p_2,q_2}}
\end{equation*}
if $$\frac1p=\frac{1}{p_1}+\frac1{p_2} \quad\text{and}\quad \frac1q=\frac{1}{q_1}+\frac{1}{q_2}.$$
\end{lem}

\subsection{Strichartz estimates}
Finally, we recall the Strichartz estimates in Lorentz spaces due to Keel and Tao \cite[Theorem 10.1]{KT}.
As we shall see, the availability of these estimates is the key ingredient in the proof of Theorem \ref{Hloc}.
\begin{prop}[\cite{KT}]\label{KT}\label{St}
Let $n\ge3.$  
Assume that $(q,r)$ and $(\tilde q , \tilde r)$ are admissible pairs, i.e.,
\begin{equation*}
\frac{2}{q}+\frac{n}{r}=\frac{n}{2}, \quad 2\leq r \leq \frac{2n}{n-2}, \quad \frac{2}{\tilde q}+\frac{n}{\tilde r}=\frac{n}{2}, \quad 2\leq \tilde r \leq \frac{2n}{n-2}.
\end{equation*}
Then we have
\begin{equation}\label{homo}
\|e^{it\Delta}f\|_{L_t^q(\mathbb{R};L_x^{r,2})}\leq C \|f\|_{L^2},
\end{equation}
\begin{equation}\label{du}
\Big\|\int_{\mathbb{R}}e^{-i\tau\Delta}F(\cdot,\tau)d\tau\Big\|_{L^2}\leq C \|F\|_{L_t^{\tilde q'}(\mathbb{R}; L_x^{\tilde r',2})}
\end{equation}
and 
\begin{equation}\label{inho}
\Big\|\int_0^t e^{i(t-\tau)\Delta} F(\cdot,\tau)d\tau\Big\|_{L_t^q(\mathbb{R};L_x^{r,2})} \leq C \|f\|_{L_t^{\tilde q'}(\mathbb{R};L_x^{{\tilde r}',2})}.
\end{equation}
\end{prop}

\section{Nonlinear estimates}\label{se3}
In this section we obtain some nonlinear estimates for the nonlinearity 
$F(u)=\lambda (I_\alpha\ast|\cdot|^{-b}|u|^p)|x|^{-b}|u|^{p-2}u$
of \eqref{HE} using the space involved in the Strichartz estimates (Proposition \ref{St}).
These nonlinear estimates will play an important role in the next section when proving the well-posedness results applying the contraction mapping principle.

\begin{prop}\label{prop}
Let $n\ge3$ and $0\leq s\le 1$. 
Assume that 
\begin{equation*}
\max\Big\{\frac{n-2}{3},n-4\Big\}<\alpha<n 
\end{equation*}
and
\begin{equation*}
\max\Big\{0, \frac{\alpha}{2}+1-\frac{(n-2s)(n+4+\sqrt{9n^2-8n+16})}{8(n-2)}\Big\}<b\leq \frac{\alpha}{2}+1-\frac{n-2s}{2}.
\end{equation*}
If the exponents $q,r$ satisfy all the conditions given as in Theorem \ref{Hloc}, then there exist certain admissible pair $(\tilde q , \tilde r)$ for which
\begin{equation}\label{non}
\big\|F(u)-F(v)\big\|_{L_t^{\tilde{q}'}(I ; L_{x}^{\tilde{r}',2})}
\leq C\big(\|u\|^{2p-2}_{L_t^{q}(I;{\dot W}^s_{r,2})}+\|v\|^{2p-2}_{L_t^{q}(I;{\dot W}^s_{r,2})}\big) \|u-v\|_{L_t^{q}(I; L_{x}^{r,2})}
\end{equation}
and 
\begin{equation}\label{non2}
\|F(u)\|_{L_t^{\tilde{q}'}(I ; {\dot W}^s_{\tilde{r}',2})}\leq C\|u\|^{2p-1}_{L_t^{q}(I;{\dot W}^s_{r,2})} 
\end{equation}
hold for $p=1+\frac{2-2b+\alpha}{n-2s}$. Here, $I=[0,T]$ denotes a finite time interval. 
\end{prop}

\begin{proof}[Proof]
Let $0< s \leq 1$.
The case $s=0$ is more simply handled by replacing Sobolev type inequalities with H\"older's inequality in the proof and so we shall omit the details. We first consider the admissible pairs $(q,r)$ and $(\tilde q,\tilde r)$ as
\begin{equation}\label{2c1}
\frac{n-2}{2}\leq \frac{n}{r} \leq \frac{n}{2}, \quad \frac{2}{q}+\frac{n}{r}=\frac{n}{2}, \quad \frac{n-2}{2}\leq \frac{n}{\tilde r} \leq \frac{n}{2}, \quad \frac{2}{\tilde q}+\frac{n}{\tilde r}=\frac{n}{2}.
\end{equation}
\subsection{Estimating \eqref{non}}
By the following simple inequality 
\begin{align*}
|F(u)-F(v)| &= \Big||x|^{-b}|u|^{p-2}u(I_\alpha\ast|x|^{-b}|u|^p)-|x|^{-b}|v|^{p-2}v(I_\alpha\ast|x|^{-b}|v|^p)|\Big|\\
&=\Big||x|^{-b}(|u|^{p-2}u-|v|^{p-2}v)(I_\alpha\ast|x|^{-b}|u|^p)\\
&\qquad\qquad\qquad\qquad\qquad + |x|^{-b}|v|^{p-2}v\big(I_\alpha \ast|x|^{-b}(|u|^p-|v|^p)\big)\Big|\\
&\leq C \Big||x|^{-b}(|u|^{p-2}+|v|^{p-2})|u-v||I_\alpha\ast|x|^{-b}|u|^p|\Big| \\
&\qquad\qquad\qquad\quad\quad +C\Big||x|^{-b}|v|^{p-1}\big(I_\alpha\ast|x|^{-b}(|u|^{p-1}+|v|^{p-1})|u-v|\big)\Big|,
\end{align*}
we need to show that the following two inequalities
\begin{align}\label{b}
\nonumber
\big\||x|^{-b}(|u|^{p-2}+|v|^{p-2})|u&-v||I_\alpha\ast|x|^{-b}|u|^p|\big\|_{L_t^{\tilde q'}(I;L_{x}^{\tilde{r}',2})} \\
&\leq C\big(\|u\|^{2p-2}_{L_t^q(I;{\dot W}^s_{r,2})}+\|v\|^{2p-2}_{L_t^q(I;{\dot W}^s_{r,2})}\big) \|u-v\|_{L_t^q(I;L_{x}^{r,2})},
\end{align}
\begin{align}\label{c}
\nonumber
\big\||x|^{-b}|v|^{p-1}\big(I_\alpha\ast|x|^{-b}(|u&|^{p-1}+|v|^{p-1})|u-v|\big)\big\|_{L_t^q(I;L_{x}^{\tilde{r}',2})} \\
&\leq C\big(\|u\|^{2p-2}_{L_t^q(I;{\dot W}^s_{r,2})}+\|v\|^{2p-2}_{L_t^q(I;{\dot W}^s_{r,2})}\big) \|u-v\|_{L_t^q(I;L_{x}^{r,2})}
\end{align}
hold for the conditions given as in Theorem \ref{Hloc}.

Now we start proving \eqref{b}.
By making use of H\"older's inequality in Lorentz spaces \eqref{Ho}, we have
\begin{align}\label{ine1}
\nonumber
\Big\||x|^{-b}(|u&|^{p-2}+|v|^{p-2})|u-v||I_\alpha\ast|x|^{-b}|u|^p|\Big\|_{L_{x}^{\tilde{r}',2}} \\
&\leq \big\||x|^{-b}(|u|^{p-2}+|v|^{p-2})|u-v|\big\|_{L_x^{r_1,\frac{4(p+1)}{p+2}}}\big\|I_\alpha\ast|x|^{-b}|u|^p\big\|_{L_x^{r_3,\frac{4(p+1)}{p}}}
\end{align}
under the following conditions
\begin{equation}\label{2c3}
\frac{1}{\tilde r'}=\frac{2p-1}{r}+\frac{2b-\alpha-2s(p-1)}{n}=\frac{1}{r_1}+\frac{1}{r_3}, 
\end{equation}
\begin{equation*}\label{2c4}
\frac{1}{r_1}=\frac{p-1}{r}+\frac{b-s(p-2)}{n}, \quad \frac{1}{r_3}=\frac{p}{r}+\frac{b-\alpha-sp}{n}, 
\end{equation*}
\begin{equation}\label{2cc2}
s+\frac{\alpha-b}{p}<\frac{n}{r}<\min\Big\{\frac{s(p-2)+n-b}{p-1}, s+\frac{n-b+\alpha}{p}\Big\}.
\end{equation}
Applying \eqref{Ho} and the Sobolev embedding in Lorentz spaces \eqref{s} to the first term in \eqref{ine1} in turn, we see
\begin{align}\label{ine2}
\nonumber
\big\||x|^{-b}(|u&|^{p-2}+|v|^{p-2})|u-v|\big\|_{L_x^{r_1,\frac{4(p+1)}{p+2}}} \\
\nonumber
&\leq \||x|^{-b}\|_{L_x^{\frac{n}{b},\infty}} \big(\|u\|_{L_x^{\frac{rn}{n-rs},4(p+1)}}^{p-2}+\|v\|_{L_x^{\frac{rn}{n-rs},4(p+1)}}^{p-2})\|u-v\|_{L_x^{r,p+1}} \\
&\lesssim \|(-\Delta)^{s/2}u\|^{p-2}_{L_x^{r,2}} \|u-v\|_{L_x^{r,2}}
\end{align}
if
\begin{equation}\label{2cc}
0<b<n\quad \text{and} \quad s<\frac{n}{r}<\frac{s(p-2)+n}{p-1}.
\end{equation}
Here, for the last inequality we used the embedding of the Lorentz spaces (see \eqref{emb}) with respect to the second index.
On the other hand, applying Lemma \ref{HLS} to the second term in \eqref{ine1} and then using \eqref{Ho} and \eqref{s} with \eqref{emb}, we see
\begin{align}\label{ine3}
\nonumber
\big\|I_\alpha\ast|x|^{-b}|u|^p\big\|_{L_x^{\frac{rn}{r(b-\alpha-sp)+pn},\frac{4(p+1)}{p}}} &\lesssim \||x|^{-b}|u|^p\|_{L_x^{\frac{rn}{r(b-sp)+pn}, \frac{4(p+1)}{p}}} \\
\nonumber
&\leq \||x|^{-b}\|_{L_x^{\frac{n}{b},\infty}} \|u\|_{L_x^{\frac{rn}{n-rs},4(p+1)}}^{p} \\
&\lesssim \|(-\Delta)^{s/2}u\|^{p}_{L_x^{r,2}}
\end{align}
if 
\begin{equation}\label{2cc3}
\max\Big\{s, s+\frac{\alpha-b}{p}\Big\}<\frac{n}{r}<s+\frac{n-b}{p}.
\end{equation}
Lastly, applying H\"older's inequality in $t$ after combining \eqref{ine1}, \eqref{ine2} and \eqref{ine3}, we obtain
\begin{align*}
\|F(u)-F(v)\|_{L_t^{\tilde{q}'}(I ; L_{x}^{\tilde{r}',2})} &\lesssim \Big\|\|u\|^{2p-2}_{{\dot W}^s_{r,2}}+\|v\|^{2p-2}_{{\dot W}^s_{r,2}}\Big\|_{L_t^{\frac{q}{2p-2}}(I)} \|u-v\|_{L_t^q(I;L_{x}^{r,2})} \\
&\lesssim \big(\|u\|^{2p-2}_{L_t^{q}(I;{\dot W}^s_{r,2})}+\|v\|^{2p-2}_{L_t^{q}(I;{\dot W}^s_{r,2})}\big) \|u-v\|_{L_t^{q}(I; L_{x}^{r,2})}
\end{align*}
with
\begin{equation}\label{con1}
\frac{1}{\tilde q'}=\frac{2p-1}{q},
\end{equation}
which implies the desired estimate \eqref{b}.

The second estimate \eqref{c} is proved similarly.
Using H\"older's inequality in Lorentz spaces \eqref{Ho}, we have 
\begin{align}\label{ine4}
\nonumber
\Big\||x|^{-b}|&v|^{p-1}\big(I_\alpha\ast|x|^{-b}(|u|^{p-1}+|v|^{p-1})|u-v|\big)\Big\|_{L_{x}^{\tilde{r}',2}}\\
&\leq \big\||x|^{-b}|v|^{p-1}\big\|_{L_x^{r_5,\frac{4(p+1)}{p-1}}} \|I_\alpha\ast|x|^{-b}(|u|^{p-1}+|v|^{p-1})|u-v|\|_{L_x^{r_7,\frac{4(p+1)}{p+3}}}
\end{align}
with
\begin{equation*}
\frac{1}{\tilde r'}=\frac{2p-1}{r}+\frac{2b-\alpha-2s(p-1)}{n}=\frac{1}{r_5}+\frac{1}{r_7},
\end{equation*}
\begin{equation*}\label{cc2}
\frac{1}{r_5}=\frac{p-1}{r}+\frac{b-s(p-1)}{n}, \quad \frac{1}{r_7}=\frac{p}{r}+\frac{b-\alpha-s(p-1)}{n},
\end{equation*}
\begin{equation}\label{2cc4}
\max\Big\{s-\frac{b}{p-1}, \frac{s(p-1)+\alpha-b}{p}\Big\}<\frac{n}{r}<\min\{s+\frac{n-b}{p-1},\frac{s(p-1)+\alpha-b+n}{p}\}
\end{equation}
For the first term in \eqref{ine4}, using \eqref{Ho} and \eqref{s} in turn, we see
\begin{equation}\label{ine5}
\big\||x|^{-b}|v|^{p-1}\big\|_{L_x^{r_5,\frac{4(p+1)}{p-1}}} \leq \||x|^{-b}\|_{L_x^{\frac{n}{b}, \infty}} \|u\|^{p-1}_{L_x^{\frac{rn}{n-rs},4(p+1)}} \lesssim \|(-\Delta)^{s/2}v\|^{p-1}_{L_x^{r,2}}
\end{equation}
if
\begin{equation}\label{2cc7}
s<\frac{n}{r}<s+\frac{n}{p-1}.
\end{equation}
Here, we used \eqref{emb} with respect to the second index for the last inequality as before.
For the second term in \eqref{ine4}, applying Lemma \ref{HLS}, \eqref{Ho} and \eqref{s} with \eqref{emb} we see
\begin{align}\label{ine6}
\nonumber
\big\|I_\alpha\ast|x|^{-b}(&|u|^{p-1}+|v|^{p-1})|u-v|\big\|_{L_x^{r_7,\frac{4(p+1)}{p+3}}}\\
\nonumber
&\lesssim \||x|^{-b} (|u|^{p-1}+|v|^{p-1})|u-v|\|_{L_x^{\frac{rn}{pn+r(b-s(p-1))},\frac{4(p+1)}{p+3}}}\\
\nonumber
&\leq \||x|^{-b}\|_{L_x^{\frac{n}{b},\infty}} \big(\|u\|_{L_x^{\frac{rn}{n-rs},4(p+1)}}^{p-1}+\|v\|_{L_x^{\frac{rn}{n-rs},4(p+1)}}^{p-1}\big)\|u-v\|_{L_x^{r,p+1}}\\
&\lesssim \big(\|(-\Delta)^{s/2}u\|_{L_x^{r,2}}^{p-1} + \|(-\Delta)^{s/2}u\|_{L_x^{r,2}}^{p-1} \big)\|u-v\|_{L_x^{r,2}}
\end{align}
if
\begin{equation}\label{2cc6}
\max\Big\{s, \frac{s(p-1)+\alpha-b}{p}\Big\} < \frac{n}{r}< \frac{s(p-1)+n-b}{p}.
\end{equation}
Now, applying H\"older's inequality in $t$ after combining \eqref{ine4}, \eqref{ine5} and \eqref{ine6} as before, we obtain the desired estimate \eqref{c}. 

\subsection{Estimating \eqref{non2}}\label{sub}
In comparison with the previous proof, we first apply Lemma \ref{FLR} to handle $(-\Delta)^{s/2}$:
\begin{align}
\nonumber
\Big\|(-\Delta)^{s/2}&\Big(|x|^{-b}|u|^{p-2}u(I_{\alpha} \ast |x|^{-b}|u|^p)\Big)\Big\|_{L_x^{\tilde r',2}} \\
\label{ine7}
&\lesssim \big\|(-\Delta)^{s/2}(|x|^{-b}|u|^{p-2}u)\big\|_{L_x^{r_1,\frac{2(2p-1)}{p-1}}} \big\|I_{\alpha}\ast|x|^{-b}|u|^p\big\|_{L_x^{r_3,\frac{2(2p-1)}{p}}} \\
\label{ine8}
&\,\,\,\quad+\big\||x|^{-b}|u|^{p-2}u\big\|_{L_x^{r_5,\frac{2(2p-1)}{p-1}}}\big\|(-\Delta)^{s/2}(I_{\alpha}\ast|x|^{-b}|u|^p)\big\|_{L_x^{r_7,\frac{2(2p-1)}{p}}}
\end{align}
with 
\begin{equation*}
\frac{1}{\tilde r'}=\frac{2p-1}{r}+\frac{2b-\alpha-2s(p-1)}{n}=\frac{1}{r_1}+\frac{1}{r_3}=\frac{1}{r_5}+\frac{1}{r_7},
\end{equation*}
\begin{equation*}
\frac{1}{r_1}=\frac{p-1}{r}+\frac{b-s(p-2)}{n}, \quad \frac{1}{r_3}=\frac{p}{r}+\frac{b-\alpha-sp}{n},
\end{equation*}
\begin{equation*}
\frac{1}{r_5}=\frac{p-1}{r}+\frac{b-s(p-1)}{n}, \quad \frac{1}{r_7}=\frac{p}{r}+\frac{b-\alpha-s(p-1)}{n},
\end{equation*}
\begin{equation}\label{r1}
s+\frac{\alpha-b}{p}<\frac{n}{r}<\min\Big\{\frac{s(p-2)+n-b}{p-1}, \frac{s(p-1)+n-b+\alpha}{p}\Big\}.
\end{equation}
Now we will show that both \eqref{ine7} and \eqref{ine8} are bounded by $$\|(-\Delta)^{s/2}u\|_{L_x^{r,2}}^{2p-1},$$
which implies the desired estimate \eqref{non2} when ${1}/{\tilde q'}=(2p-1)/q.$

In order to bound \eqref{ine7}, we first apply Lemma \ref{FLR} to the first term to see
\begin{align}\label{ine}
\nonumber
\big\|&(-\Delta)^{s/2}(|x|^{-b}|u|^{p-2}u)\big\|_{L_x^{r_1,\frac{2(2p-1)}{p-1}}} \\
\nonumber
&\quad \leq \|(-\Delta)^{s/2}|x|^{-b}\|_{L_x^{\frac{n}{b+s},\infty}}\|u\|_{L_x^{\frac{rn}{n-rs},2(2p-1)}}^{p-1} \\
\nonumber
&\qquad \qquad \quad \quad+ \||x|^{-b}\|_{L_x^{\frac{n}{b},\infty}}\|(-\Delta)^{s/2}(|u|^{p-2}u)\|_{L_x^{\frac{rn}{n(p-1)-rs(p-2)},\frac{2(2p-1)}{p-1}}}\\
&\quad=: I+II
\end{align}
if 
\begin{equation}\label{2cc8}
s<\frac{n}{r}<\frac{s(p-2)+n}{p-1}.
\end{equation}
For the first part $I$, using the Sobolev embedding \eqref{s} with \eqref{emb}, we have
\begin{align}
\nonumber
\|(-\Delta)^{s/2}|x|^{-b}\|_{L_x^{\frac{n}{b+s},\infty}}\|u\|_{L_x^{\frac{rn}{n-rs},2(2p-1)}}^{p-1} &\lesssim \||x|^{-b-s}\|_{L_x^{\frac{n}{b+s},\infty}} \|(-\Delta)^{s/2}u\|_{L_x^{r,2}}^{p-1} \\
\label{ine10}
&\lesssim \|(-\Delta)^{s/2}u\|_{L_x^{r,2}}^{p-1}
\end{align}
when
\begin{equation}\label{cb}
0<b<n-s.
\end{equation}
On the other hand, we apply Lemma \ref{FCR} to the second part $II$ as follows:
\begin{align}\label{ine9}
\nonumber
\||x|^{-b}\|_{L_x^{\frac{n}{b},\infty}}\|&(-\Delta)^{s/2}(|u|^{p-2}u)\|_{L_x^{\frac{rn}{n(p-1)-rs(p-2)},\frac{2(2p-1)}{p-1}}}\\
&\quad\lesssim \|u\|_{L_x^{\frac{rn}{n-rs}, \frac{2(2p-1)}{p-2}}}^{p-2}\|(-\Delta)^{s/2}u\|_{L_x^{r,\frac{2(2p-1)}{3}}}\lesssim \|(-\Delta)^{s/2}u\|_{L_x^{r,2}}^{p-1}.
\end{align}
Here, for the second inequality in \eqref{ine9}, we applied \eqref{s} after using \eqref{emb}.

For the second term in \eqref{ine7}, we make use of Lemma \ref{HLS}:
\begin{align}
\nonumber
\big\|I_{\alpha}\ast|x|^{-b}|u|^p\big\|_{L_x^{r_3,\frac{2(2p-1)}{p}}} &\lesssim \||x|^{-b}|u|^p\|_{L_x^{\frac{rn}{pn+r(b-sp)},\frac{2(2p-1)}{p}}} \\
\nonumber
&\lesssim \||x|^{-b}\|_{L_x^{\frac{n}{b},\infty}} \|u\|_{L_x^{\frac{rn}{n-rs},2(2p-1)}}^{p} \\
\label{ine11}
&\lesssim \|(-\Delta)^{s/2}u\|_{L_x^{r,2}}^{p}
\end{align}
if
\begin{equation}\label{2cc9}
\max\Big\{s, s+\frac{\alpha-b}{p}\Big\}<\frac{n}{r}<s+\frac{n-b}{p}.
\end{equation}
Here, we used \eqref{Ho} for the second inequality, and \eqref{s} with \eqref{emb} for the last one.
Thus, by combining \eqref{ine}, \eqref{ine10}, \eqref{ine9} and \eqref{ine11}, the estimate \eqref{ine7} is bounded by $\|(-\Delta)^{s/2}u\|_{L_x^{r,2}}^{2p-1}$.

Finally, we bound \eqref{ine8}.
By using H\"older's inequality \eqref{Ho} and the Sobolev embedding \eqref{s} with \eqref{emb} in turn, the first term in \eqref{ine8} is bounded as 
\begin{equation}\label{ineq12}
\big\||x|^{-b}|u|^{p-2}u\big\|_{L_x^{r_5,\frac{2(2p-1)}{p-1}}} \leq \||x|^{-b}\|_{L_x^{\frac{n}{b},\infty}} \|u\|_{L_x^{\frac{rn}{n-rs},2(2p-1)}}^{p-1}\lesssim \|(-\Delta)^{s/2}u\|_{L_x^{r,2}}^{p-1}
\end{equation}
if 
\begin{equation}\label{2cc11}
s<\frac{n}{r}<s+\frac{n}{p-1}.
\end{equation}
On the other hand, applying Lemma \ref{HLS}, \eqref{Ho} and \eqref{s} with \eqref{emb}, the second term in \eqref{ine8} is bounded by
\begin{align}\label{ineq13}
\nonumber
\big\|(-\Delta)^{s/2}(I_{\alpha}\ast|x|^{-b}|u|^p)\big\|_{L_x^{r_7,\frac{2(2p-1)}{p}}} &= \big\||x|^{\alpha-s-n}\ast|x|^{-b}|u|^p\big\|_{L_x^{r_7,\frac{2(2p-1)}{p}}} \\
\nonumber
&\lesssim \||x|^{-b}|u|^p\|_{L_x^{\frac{rn}{np+r(b-sp)},\frac{2(2p-1)}{p}}} \\
\nonumber
& \leq \||x|^{-b}\|_{L_x^{\frac{n}{b},\infty}} \|u\|_{L_x^{\frac{rn}{n-rs},2(2p-1)}}^{p} \\
& \lesssim \|(-\Delta)^{s/2}u\|_{L_x^{r,2}}^{p},
\end{align}
if 
\begin{equation}\label{2cc10}
s<\alpha<n \quad \text{and} \quad \max\Big\{s, \frac{(p-1)s+\alpha-b}{p}\Big\}<\frac{n}{r}<s+\frac{n-b}{p}.
\end{equation}
Therefore, combining \eqref{ineq12} and \eqref{ineq13}, we have the bound $\|(-\Delta)^{s/2}u\|_{L_x^{r,2}}^{2p-1}$.

\subsection{Deriving the assumptions}
It remains to check the assumptions under which \eqref{non} and \eqref{non2} hold.
Inserting \eqref{con1} into the last condition in \eqref{2c1} implies 
$$\frac{2(2p-1)}{q}=2-\frac{n}{2}+\frac{n}{\tilde r}.$$
Substituting this into the second condition in \eqref{2c1} implies 
\begin{equation}\label{2c}
\frac{1}{\tilde r} = p-\frac{2}{n}-\frac{2p-1}r. 
\end{equation}
Note that \eqref{2c} is exactly same as the condition \eqref{2c3} when $p=1+\frac{2-2b+\alpha}{n-2s}$.
Substituting \eqref{2c} into the third condition in \eqref{2c1} is also implies 
\begin{equation}\label{2cc1}
\frac{n}{2}-\frac{2}{2p-1} \leq \frac{n}{r} \leq \frac{n}{2}-\frac{1}{2p-1},
\end{equation}
by which the first condition in \eqref{2c1} can be eliminated since $p\ge2$.

From the conditions \eqref{2cc2}, \eqref{2cc} and \eqref{2cc3} which are occurred in the proof of \eqref{b}, we have 
\begin{equation}\label{r2}
\max\Big\{s, s+\frac{\alpha-b}{p}\Big\}<\frac{n}{r}<\min\Big\{\frac{s(p-2)+n-b}{p-1}, s+\frac{n-b}{p}\Big\}.
\end{equation}
From the conditions \eqref{2cc4}, \eqref{2cc7} and \eqref{2cc6} in the proof of \eqref{c}, we also have
\begin{equation}\label{r3}
\max\Big\{s, \frac{s(p-1)+\alpha-b}{p}\Big\}<\frac{n}{r}<\frac{s(p-1)+n-b}{p}.
\end{equation}
The lower bound of $n/r$ in \eqref{r3} can be eliminated by the lower one of $n/r$ in \eqref{r2}.
Also, the upper bound of $n/r$ in \eqref{r2} is redundant by the upper one of $n/r$ in \eqref{r3} since $0<b<n-s$.
By combining these two conditions, we get
\begin{equation}\label{rr2}
\max\Big\{s, s+\frac{\alpha-b}{p}\Big\} < \frac{n}{r}< \frac{s(p-1)+n-b}{p}.
\end{equation}
Similarly, from the conditions \eqref{r1}, \eqref{2cc8}, \eqref{2cc9}, \eqref{2cc10} and \eqref{2cc11} arisen in Subsection \ref{sub}, we get
\begin{equation}\label{rr3}
\max\{s,s+\frac{\alpha-b}{p}\}<\frac{n}{r}<\frac{s(p-2)+n-b}{p-1}
\end{equation}
which is also eliminated by \eqref{rr2} since $0<b<n-s$.
Combining \eqref{2cc1} and \eqref{rr2}, we then get 
\begin{equation}\label{rr}
\max\Big\{s, s+\frac{\alpha-b}{p}, \frac{n}{2}-\frac{2}{2p-1}\Big\}< \frac{n}{r}<\min\Big\{\frac{s(p-1)+n-b}{p}, \frac{n}{2}-\frac{1}{2p-1}\Big\},
\end{equation}
which implies the assumption \eqref{as-r}.

Now we derive the assumption \eqref{as-1}.
We first make the lower bound of $n/r$ less than the upper one of $n/r$ in \eqref{rr}.
Then we have 
\begin{equation}\label{r4}
b+s<n, \quad s<\frac{n}{2}-\frac{1}{2p-1}
\end{equation}
\begin{equation}\label{r5}
\alpha+s<n, \quad ps+\alpha-b<\frac{pn}{2}-\frac{p}{2p-1},\quad b+\frac{(p-2)n}{2}<(p-1)s+\frac{2p}{2p-1}. 
\end{equation}
Here, the conditions in \eqref{r4} can be eliminated by \eqref{cb} and $s\leq 1$, respectively.
Inserting $s=\frac{n}{2}-\frac{2-2b+\alpha}{2(p-1)}$ into \eqref{r5}, we see 
\begin{equation}\label{e}
2b<-(2p-3)\alpha+2+(p-1)n, \quad 2b<-(p-2)\alpha+\frac{2p^2}{2p-1}, \quad \alpha<n+\frac{2}{2p-1}.
\end{equation}
Here, the upper bound of $\alpha$ in \eqref{e} can be eliminated by $\alpha<n$. 
On the other hand, substituting $s=\frac{n}{2}-\frac{2-2b+\alpha}{2(p-1)}$ into the following conditions (see \eqref{cb} and \eqref{2cc10})
\begin{equation*}
s<\alpha<n, \quad 0<b<n-s, \quad 0\leq s \leq 1
\end{equation*} 
implies 
\begin{equation}\label{b1}
\begin{gathered}
2b<(2p-1)\alpha+2-(p-1)n, \quad \alpha<n, \quad 0<2b<\frac{\alpha+(p-1)n+2}{p},\\
\alpha+2-(p-1)n \leq 2b \leq \alpha+2p-(p-1)n.
\end{gathered}
\end{equation}
Here, the second upper bound of $2b$ in \eqref{b1} is redundant by the last upper one of $2b$ in \eqref{b1}.
Making the lower bound of $2b$ less than the upper one of $2b$ in \eqref{e} and \eqref{b1}, we have
\begin{equation}\label{p1}
\begin{gathered}
(2p-3)\alpha<(p-1)n+2, \quad (p-2)\alpha<\frac{2p^2}{2p-1}, \quad (p-1)n<(2p-1)\alpha+2,\\
(p-1)n<\alpha+2p, \quad 0<\alpha<n,\quad \alpha<n+\frac{2(p-1)}{2p-1}.
\end{gathered}
\end{equation}
The last condition in \eqref{p1} can be eliminate the fifth one in \eqref{p1}.

To eliminate $p$, we rewrite \eqref{p1} with respect to $p$ as follows:
\begin{equation}\label{p}
\begin{cases}
2<p<\frac{n-\alpha+2}{n-2\alpha} \quad \text{if} \quad 0<\alpha\leq 1\\
2<p<\min\{\frac{n-\alpha+2}{n-2\alpha},\frac{5\alpha+\sqrt{9\alpha^2+16\alpha}}{4(\alpha-1)} \} \quad \text{if} \quad 1< \alpha\leq \frac{n}{2}\\
2<p<\min\{\frac{n-3\alpha-2}{n-2\alpha},\frac{5\alpha+\sqrt{9\alpha^2+16\alpha}}{4(\alpha-1)} \}\quad \text{if} \quad \frac{n}{2}< \alpha<n
\end{cases}
\end{equation}
and 
\begin{equation}\label{p2}
p<\frac{n+\alpha}{n-2}.
\end{equation}
Making the lower bounds of $p$ less than the upper ones of $p$ in \eqref{p} and \eqref{p2}, we get 
$$\max\Big\{n-4,\frac{n-2}{3}\Big\}<\alpha<n$$
which is the assumption \eqref{as-1}.
Indeed, when $n=3,4,5$, the lower bound $(n-2)/3$ of $\alpha$ is determined by $2<(n-\alpha+2)/(n-2\alpha)$ if $0<\alpha\leq n/2$. If $n/2<\alpha<n$, the lower bound of $\alpha$ becomes $n-4$ by $2<(n+\alpha)/(n-2)$, but which is redundant by the lower bound $n/2$ of $\alpha$.
When $n\ge6$, the lower bound $n-4$ of $\alpha$ comes from $2<(n+\alpha)/(n-2)$ if $1<\alpha<n$.

Finally, it remains to derive the assumption \eqref{as-2}. To do this, we write \eqref{p1} with respect to $\alpha$:
\begin{equation*}
\max\Big\{\frac{(p-1)n-2}{2p-1}, (p-1)n-2p\Big\}<\alpha<\min\Big\{n,\frac{2p^2}{(p-2)(2p-1)},\frac{(p-1)n+2}{2p-3}\Big\}.
\end{equation*}
Making the lower bound of $p$ less than the upper one of $p$, we have 
\begin{equation}\label{p3}
2\leq p<\frac{5n-4+\sqrt{9n^2-8n+16}}{4(n-2)}.
\end{equation}
Substituting $p=1+\frac{2-2b+\alpha}{n-2s}$ into \eqref{p3}, we see
\begin{equation*}
\frac{\alpha}{2}+1-\frac{(n-2s)(n+4+\sqrt{9n^2-8n+16})}{8(n-2)}<b\leq \frac{\alpha}{2}+1-\frac{n-2s}{2}
\end{equation*}
which implies the assumption \eqref{as-2}.
\end{proof}

\section{The well-posedness}\label{se4}
This section is devoted to proving Theorems \ref{Hloc} and \ref{Hglo} by the contraction mapping principle.
The nonlinear estimates just obtained above play an important role in this step.  
By Duhamel's principle, we first write the solution of the Cauchy problem \eqref{HE} as
\begin{equation}\label{2Duhamel1}
\Phi(u)=\Phi_{u_0}(u) = e^{it \Delta} u_0 - i\lambda  \int_{0}^{t} e^{i(t-\tau)\Delta}F(u)\, d\tau
\end{equation}
where $F(u)= \big(I_\alpha\ast|\cdot|^{-b}|u(\cdot,\tau)|^p\big)|\cdot|^{-b}|u(\cdot,\tau)|^{p-2}u(\cdot,\tau)$.
For appropriate values of  $T,M,N>0$ determined later, we shall show that
$\Phi$ defines a contraction map on
\begin{align*}
X(T,M,N)=\Big\lbrace u\in  C_t(I;H^s)\cap L_t^q(I;W_{r,2}^s):\sup_{t\in I}\|u\|_{H^s}\leq M,\ \|u\|_{L_t^q(I;W_{r,2}^s)} \leq N \Big\rbrace
\end{align*}
equipped with the distance
$$d(u,v)=\|u-v\|_{L_t^q(I;L_x^{r,2})}$$
where $I=[0,T]$ and the exponents $q, r$ are given as in Theorem \ref{Hloc}.

We first show that $\Phi$ is well-defined on $X$, i.e. $\Phi(u)\in X$ for $u\in X$.
By applying Plancherel's theorem, \eqref{du} and the nonlinear estimates \eqref{non} and \eqref{non2} with
\begin{equation}\label{2se1}
\frac{1}{\tilde{q}'}=\frac{2p-1}{q} \quad \text{and} \quad \frac{1}{\tilde r'}=\frac{2p-1}{r}+\frac{2b-\alpha-2s(p-1)}{n},
\end{equation}
we have
\begin{align}\label{2fg11}
\nonumber
\sup_{t\in I}\|\Phi(u)\|_{H^s}
&\leq C \| u_0 \|_{H^s} +C\sup_{t\in I} \Big\| \int_0^{t}  e^{-i\tau\Delta} \chi_{[0,t]}(\tau) F(u)\, d\tau \Big\|_{H^s}\\
\nonumber
&\leq C \| u_0 \|_{H^s}+C\|F(u)\|_{L_t^{{\tilde q}'}(I;L_x^{{\tilde r}',2})}+C\|(-\Delta)^{s/2} F(u)\|_{L_t^{{\tilde q}'}(I;L_x^{{\tilde r}',2})}\\
&\leq C \| u_0 \|_{H^s} +C\|  u \|^{2p-1}_{L_t^{q}(I;W_{r,2}^s) }.
\end{align}
Here, for the second inequality we also used that $\|f\|_{H^s} \lesssim \|f\|_{L^2}+ \|f\|_{\dot H^s}.$
On the other hand, by using \eqref{inho}, \eqref{non} and \eqref{non2} under the relation \eqref{2se1}, we see
\begin{align}\label{2fg33}
\nonumber
&\|\Phi(u)\|_{L_t^q(I;W_{r,2}^s)}\\
&\,\,\,\,\,\leq \| e^{it \Delta} u_0\|_{L_t^q(I;W_{r,2}^s)} + C\| F(u) \|_{L_t^{\tilde q'}(I;L_x^{\tilde r',2})}+ C\|(-\Delta)^{s/2} F(u) \|_{L_t^{\tilde q'}(I;L_x^{\tilde r',2})}\nonumber\\
&\,\,\,\,\,\leq \| e^{it \Delta} u_0\|_{L_t^q(I;W_{r,2}^s)} + C\| u \|^{2p-1}_{L_t^{q}(I;W_{r,2}^s)}.
\end{align}
By the dominated convergence theorem, we take here $T>0$ small enough so that
\begin{equation}\label{2as1}
\|e^{it\Delta} u_0\|_{L_t^q(I;W_{r,2}^s)} \leq \varepsilon
\end{equation}
for some $\varepsilon>0$ chosen later.
From \eqref{2fg11} and \eqref{2fg33}, it follows that
\begin{equation*}
\sup_{t\in I}\|\Phi(u)\|_{H^s}\leq C\|u_0 \|_{H^s}+ CN^{2p-1} \quad \text{and}\quad \|\Phi(u)\|_{L_t^q(I;W_{r,2}^s)}\leq \varepsilon+ CN^{2p-1}
\end{equation*}
for $u \in X$.
Therefore, $\Phi(u)\in X$ if
\begin{equation}\label{2az1}
C\|u_0 \|_{H^s}+ CN^{2p-1}\leq M \quad \text{and}\quad \varepsilon+ CN^{2p-1} \leq N.
\end{equation}

Next we show that $\Phi$ is a contraction on $X$, i.e., for $u,v\in X$ 
$$d(\Phi(u), \Phi(v))\leq \frac12 d(u,v).$$
Using \eqref{inho} and \eqref{non} by the similar argument employed to show \eqref{2fg33}, it follows that, for $u, v \in X$
\begin{align*}
d(\Phi(u), \Phi(v))&= \|\Phi(u)-\Phi(v)\|_{L_{t}^{q}(I; L_x^{r,2})} \\
&\leq 2C\|F(u)-F(v)\|_{L_t^{\tilde q'}(I;L_x^{\tilde r',2})} \\
&\leq 2C (\| u\|_{L_t^{q}(I;W_{r,2}^s)}^{2p-2} +\| v\|_{L_t^{q}(I;W_{r,2}^s)}^{2p-2})\|u-v\|_{L_t^{q}(I;L_x^{r,2})} \nonumber\\
&\leq 4CN^{2p-2} d(u,v)
\end{align*}
under the relation \eqref{2se1}.
Now by setting $M=2C\|u_0\|_{L^2}$ and $N=2\varepsilon$ for $\varepsilon>0$ small enough so that
\eqref{2az1} holds and $4CN^{2p-2} \leq1/2$,
it follows that $X$ is stable by $\Phi$ and $\Phi$ is a contraction on $X$.
Therefore, by the contraction mapping principle, there exists a unique local solution $u \in C(I ; H^s) \cap L_t^{q}(I ; W_{r,2}^s)$.

The continuous dependence of the solution $u$ with respect to initial data $u_0$ follows obviously in the same way;
if $u,v$ are the corresponding solutions for initial data $u_0,v_0$, respectively, then
\begin{align*}
d(u,v)&\leq d\big(e^{it \Delta}u_0, e^{it\Delta} v_0\big) +d\bigg(\int_{0}^{t} e^{i(t-\tau)\Delta}F(u)d\tau, \int_{0}^{t} e^{i(t-\tau)\Delta}F(v) d\tau\bigg)\\
&\leq C\|u_0-v_0\|_{L^2}+C\|F(u)-F(v)\|_{L_t^{\tilde q'}(I; L_x^{\tilde r', 2})}\\
&\leq C\|u_0-v_0\|_{L^2}+\frac12\|u-v\|_{L_t^{q}(I; L_{x}^{r,2})}
\end{align*}
under the relation \eqref{2se1}, which implies $d(u,v) \lesssim \|u_0-v_0\|_{H^s}$.

Since $u_0 \in H^s$, by Proposition \ref{KT}, the smallness condition \eqref{2as1} can be replaced by that of $\|u_0\|_{H^s}$ as
$$\|e^{it\Delta}u_0\|_{L_t^{q}(I;W_{r,2}^s)} \leq C\|u_0\|_{H^s} \leq \varepsilon$$
from which we can choose $T=\infty$ in the above argument to get the global unique solution.
It only remains to prove the scattering property.
Following the argument above, one can easily see that
\begin{align*}
\big\|e^{-it_2\Delta}u(t_2)-e^{-it_1\Delta}u(t_1)\big\|_{H^s}&=\Big\|\int_{t_1}^{t_2}e^{-i\tau\Delta}F(u)d\tau\Big\|_{H^s}\\
&\lesssim\|F(u)\|_{L_t^{\tilde q'}([t_1, t_2]; L_x^{\tilde r',2})}+\||\nabla|^s F(u)\|_{L_t^{\tilde q'}([t_1, t_2]; L_x^{\tilde r',2})}\\
&\lesssim\| u \|_{L_t^q([t_1,t_2]; W_{r,2}^s)}^{2p-1} \,\,\rightarrow\,\,0
\end{align*}
as $t_1,t_2\rightarrow\infty$. This implies that
$\varphi:=\lim_{t\rightarrow\infty}e^{-it\Delta}u(t)$ exists in $H^s$.
Moreover, 
$$u(t)-e^{it\Delta}\varphi= i\lambda  \int_{t}^{\infty} e^{i(t-\tau)\Delta}F(u) d\tau,$$
and hence
\begin{align*}
\big\|u(t)-e^{it\Delta}\varphi\big\|_{H^s}&\lesssim \bigg\|\int_{t}^{\infty} e^{i(t-\tau)\Delta}F(u) d\tau\bigg\|_{H^s}\\
&\lesssim\|F(u)\|_{L_t^{\tilde q'}([t,\infty); L_x^{\tilde r',2})}+\|(-\Delta)^{s/2}F(u)\|_{L_t^{\tilde q'}([t,\infty); L_x^{\tilde r',2})}\\
&\lesssim\| u \|_{L_t^{q}([t,\infty); W_{r,2}^s)}^{2p-1} \quad\rightarrow\quad0
\end{align*}
as $t\rightarrow\infty$.
This completes the proof.


\begin{thebibliography}{9}
\bibitem{AS}
M. G. Alharbi and T. Saanouni, \textit{Sharp threshold of global well-posedness vs finite time blow-up for a class of inhomogeneous Choquard equations} J. Math. Phys. 60 (2019), 081514, 24 pp.

\bibitem{AT} L. Aloui and S. Tayachi, \textit{Local well-posedness for the inhomogeneous nonlinear Schr\"odinger equation}, Discrete Contin. Dyn. Syst. 41 (2021), 5409-5437.

\bibitem{CHO} Y. Cho, G. Hwang and T. Ozawa,
\textit{Global well-posedness of critical nonlinear Schr\"odinger equations below $L^2$}, Discrete Contin. Dyn. Syst. 33 (2013), 1389-1405.

\bibitem{CUN} D. Cruz-Uribe and V. Naibo, \textit{Kato-Ponce inequalities on weighted and variable Lebesgue spaces}, Differential Integral Equations 29 (2016), 801–836.

\bibitem{GW} Y. Gao and Z. Wang, \textit{Scattering versus blow-up for the focusing $L^2$ supercritical Hartree equation}, Z. Angew. Math. Phys. 65 (2014), 179–202.

\bibitem{GM} E.P. Gross and E. Meeron \textit{Physics of Many-Particle System}, Gordon Breach, New York (1996), vol. 1, pp. 231-406. 

\bibitem{KT} M. Keel and T. Tao, \textit{Endpoint Strichartz estimates}, Amer. J. Math. 120 (1998), 955-980.

\bibitem{KLS1} S. Kim, Y. Lee and I. Seo, \textit{Sharp weighted Strichartz estimates and critical inhomogeneous Hartree equations}, arXiv: 2110.14922.

\bibitem{LR} P. G. Lemarié-Rieusset, \textit{Recent developments in the Navier-Stokes problem}, Chapman $\&$ Hall/CRC Research Notes in Mathematics, 431. Chapman $\&$ Hall/CRC, Boca Raton, FL, 2002. xiv+395 pp.

\bibitem{MXZ}
C. Miao, G. Xu abd L. Zhao, \textit{Global well-posedness and scattering for the energy-critical, defocusing Hartree equation for radial data}, J. Funct. Anal. 253 (2007), 605–627.

\bibitem{MPT} I.M. Moroz, R. Penrose and P. Tod, \textit{Spherically-symmetric solutions of the Schrödinger-Newton equations}, Classical Quantum Gravity 15 (1998), 2733–2742.

\bibitem{Saa} T. Saanouni, \textit{Scattering threshold for the focusing Choquard equation},
NoDEA Nonlinear Differential Equations Appl. 26 (2019), Paper No. 41, 32 pp.

\bibitem{SA} T. Saanouni and T. Alharbi, \textit{On the inter-critical inhomogeneous generalizaed Hartree equation}, Arab. J. Math., (2022).

\bibitem{SP} T. Saanouni and C. Peng, \textit{Scattering for a radial defocusing inhomogeneous Choquard equation} Acta Appl. Math. 177 (2022), Paper No. 6, 14 pp.

\bibitem{SX} T. Saanouni and C. Xu, \textit{Scattering theory for a class of radial focusing inhomogeneous Hartree equations}, Potential Anal. 58 (2023), 617-643.

\bibitem{X} C. Xu, \textit{Scattering for the non-radial focusing inhomogeneous nonlinear Schr\"odinger-Choquard equation}, arXiv:2104.09756.




\end{thebibliography}
\end{document}